\documentclass{article}
\usepackage[margin = 1in]{geometry} 
\usepackage{graphicx}
\usepackage{amsmath}
\usepackage{amsfonts}
\usepackage{amsthm}
\usepackage{amssymb}
\usepackage[hidelinks]{hyperref}
\usepackage[affil-it]{authblk}
\usepackage{enumitem}
\usepackage{multirow}

\newtheorem{theorem}{Theorem}[section]
\newtheorem{lemma}[theorem]{Lemma}

\newtheorem{defn}[theorem]{Definition} 
\newtheorem{corollary}[theorem]{Corollary} 

\DeclareMathOperator{\arcsinh}{arcsinh}
\DeclareMathOperator{\Tr}{Tr}
\DeclareMathOperator{\Conv}{Conv}
\DeclareMathOperator{\Vol}{Vol}
\DeclareMathOperator{\dist}{dist}

\begin{document}

\title{Egan conjecture holds}
\author{Sergei Drozdov\footnote{email: \href{mailto:smdrozdov@gmail.com}{smdrozdov@gmail.com} }}

\maketitle
\begin{abstract}
Given a Euclidean simplex of dimension $n\geqslant 2$ let its radii of inscribed and circumscribed spheres be $r$ and $R$, and the distance between the centers of the inscribed and circumscribed spheres be $d.$ Then,
\begin{equation*}
(R-nr)(R+(n-2)r) \geqslant d^2.
\end{equation*}

\end{abstract}
\section{Introduction}
Given a triangle with radii of inscribed and circumscribed circles of $r$ and $R$, and the distance between the centers of the inscribed and circumscribed circles being $d$ a classical result of Chapple (1746) and Euler (1765) is:
\begin{equation}
R^2 - 2Rr = d^2.
\label{eq:chapple}
\end{equation}
The similar result for the Euclidean tetrahedron was proved independently by Grace (1918) and Danielsson (1952). Namely, if we replace the inscribed and circumscribed circles with spheres, the following inequality holds:
\begin{equation*}
(R-3r)(R+r) \geqslant d^2.
\end{equation*}

In 2014 Greg Egan conjectured in \cite{Conjecture}, that if $n \geqslant 4,$ for an $n$-dimensional Euclidean simplex the inequality takes the form:
\begin{equation}
(R-nr)(R+(n-2)r) \geqslant d^2,
\label{eganconjecture}
\end{equation}
and showed that this condition is sufficient.

In 2017 Cho and Naranjo proved in \cite{Inequality} the version of equation \eqref{eq:chapple} for a spherical triangle in $S^2:$

\begin{equation}
\sin^2 (R-r) - \sin^2r \cos^2 R=\sin^2 d.
\label{spherical}
\end{equation}

For any integer $n \geqslant 2$ we prove an inequality (Theorem \ref{th:main}), which is an analogue of both inequality \eqref{eganconjecture} and equation \eqref{spherical} for a spherical simplex in $S^n.$

It leads to inequality \eqref{eganconjecture} (Theorem \ref{th:ecproof}) for a Euclidean simplex in any dimension.
\footnote{Data sharing not applicable to this article as no datasets were generated or analysed during the current study.}

\section{The Main Result}

The inequalities above are related to a simplex and its two spheres: inscribed and circumscribed. This leads to a certain asymmetry in the notation. The spherical simplices - in contrast to Euclidean ones - have a convenient concept of \textit{polar simplex}. 
The circumscribed sphere of a simplex is complementary to the inscribed sphere of the simplex's polar (they have the same center and sum of their radii is $\frac{\pi}{2}$).

\begin{defn} \textbf{Spherical Simplex} 
Given $n \geqslant 3$, let $u_1, \dots, u_n$ be distinct vectors in the unit sphere $S^{n-1}$. Then, the set $U := \{u_i\}$ defines a spherical simplex, which has $n$ vertices.
\end{defn}

\begin{defn} \textbf{Center and Radius of a Spherical Simplex's Circumscribed Sphere} 
For a spherical simplex $U = \{u_i\}$, we define $t$ and $\alpha$ as the center and radius of the circumscribed sphere, respectively, if $t \in S^{n-1}$ and all $u_i$ satisfy $u_i \cdot t = \cos \alpha > 0$.
\end{defn}

\begin{defn} \textbf{Polar Spherical Simplices} 
Two spherical simplices $U = \{u_i\}$ and $V = \{v_i\}$ are said to be mutually polar if for all $i$ and $j$, where $1 \leqslant i < j  \leqslant n$, the vectors $u_i$ and $v_j$ are orthogonal ($u_i \cdot v_j = 0$), and $u_i \cdot v_i > 0$ for all $1 \leqslant i \leqslant n$.
\end{defn}

We consider the Theorem \ref{adjacent} known:
\begin{theorem}
\label{adjacent}
Given two mutually polar spherical simplices $U$ and $V,$ let the circumscribed sphere of $U$ have center and radius $O_u, \beta,$ and inscribed sphere of $V$ have center and radius $I_v, B$ Then:
\begin{align*}
\beta+B = \frac{\pi}{2}, \quad O_u = I_v.
\end{align*}
\end{theorem}
Let's prove the main result of this paper:
\begin{theorem}
\label{th:main}
Given two mutually polar spherical simplices $U$ and $V$ in $S^{n-1}$ $(n \geqslant 3)$, let $\alpha$ be the distance between their circumscribed sphere centers, $\beta$ and $\gamma$ be their respective circumscribed sphere radii, and $b = \tan \beta$, $c = \tan \gamma$. Then, the following inequality holds:
\begin{equation*}
\sqrt{ (bc-1)^2  \cos ^ 2 \alpha - (b + c)^2 \sin^2 \alpha} \geqslant n - 2.
\end{equation*}
In the case where $n = 3$, the inequality becomes an equality.
\end{theorem}
\begin{proof}
\begin{enumerate}[leftmargin = *]

\item \textbf{Three weak inequalities:}
First, by construction 
\begin{equation}\frac{\pi}{2} > \beta, \gamma > 0.\label{eq:betagamma}\end{equation} 
Now let's prove a Lemma
\begin{lemma}
In the notation as above
\begin{equation}
\beta + \gamma > \alpha + \frac{\pi}{2}.
\label{eq:betagamma2}
\end{equation}
\end{lemma}
\begin{proof}
Let $O, I$ be the circumcenters of $U, V$. Let point $t$ be any point such that $\angle(t,I) = \frac{\pi}{2} - \gamma.$  We may say that $t$ lies on the \textit{inscribed} sphere of the simplex $U$.

We will prove that $t,$ and the whole inscribed sphere of $U$ lies inside the circumscribed sphere of $U.$

For any $i = 1,\dots,n:$
\begin{equation*}
\angle(t,v_i) \leqslant \angle(t,I) + \angle(I,v_i) = \gamma + \left(\frac{\pi}{2} - \gamma\right) = \frac{\pi}{2},
\end{equation*}
and the equality is only attained if $t,I,v_i$ lie on the same big circle. So, the equality can be attained for not more than one $i.$ Without any loss of generality, lets assume:
\begin{align*}
\begin{cases}
\angle(t,v_i) < \frac{\pi}{2}, i = 1,\dots,n-1\\
\angle(t,v_n) \leqslant \frac{\pi}{2}.
\end{cases}
\end{align*}
Vectors $u_i,$ vertices of $U,$ are a basis in $\mathbb{R}^n,$ so we can expand $t$:
\begin{equation*}
t = x_1u_1 + \dots + x_nu_n, \text{ for some } x_i \in \mathbb{R}.
\end{equation*}
Multiplying by $v_i$ for some $i$ we get:
\begin{equation*}
tv_i = x_i \cdot u_iv_i, \text{ so}
\end{equation*}
\begin{equation*}
x_i \cdot u_iv_i = \cos\left(\angle(t,v_i)\right)
\end{equation*}

As $u_iv_i>0, \forall i$ we get:

\begin{align*}
&x_i >0, i = 1,\dots,n-1\\
&x_n \geqslant 0.
\end{align*}

As $|t| = 1,$ we can expand:
\begin{align*}
1& = (x_1u_1 + \dots + x_nu_n)^2 = \sum\limits_{1\leqslant i, j \leqslant n} x_ix_j  \cdot u_iu_j = x_1x_2 \cdot u_1u_2+\sum\limits_{(i,j) \neq (1,2)} x_ix_j  \cdot u_iu_j\leqslant\\
&\leqslant x_1x_2 \cdot u_1u_2+\sum\limits_{(i,j) \neq (1,2)} x_ix_j<x_1x_2+\sum\limits_{(i,j) \neq (1,2)} x_ix_j = (x_1+\dots+x_n)^2.\\
\end{align*}

So, $x_1+\dots+x_n > 1.$
Now let's expand $t\cdot O$

\begin{equation*}
t\cdot O = x_1 u_1O + \dots + x_n u_nO = (x_1 + \dots +x_n) \cos\beta > \cos\beta.
\end{equation*}

So, 
\begin{equation*}
\angle(t,O) < \beta.
\end{equation*}

This holds for any point $t$ on the sphere of radius $\frac{\pi}{2}-\gamma,$ with the center $I.$ As $\angle(O,I) = \alpha,$ we get:

\begin{equation*}
\alpha + \left(\frac{\pi}{2}-\gamma\right)<\beta.
\end{equation*}
\end{proof}

Combining inequalities \eqref{eq:betagamma} and \eqref{eq:betagamma2} we get:
\begin{equation}
\pi > \beta + \gamma > \alpha + \frac{\pi}{2}. \label{eq:ineq1}
\end{equation}
So,
\begin{equation}
\frac{\pi}{2} > \alpha,\beta,\gamma > 0, \cos \alpha >0, \text{ and} \label{eq:ineq2}
\end{equation}

\begin{equation}
bc = \tan\beta \tan\gamma > \tan\beta\tan\left(\frac{\pi}{2}-\beta\right) = 1. \label{eq:ineq3}
\end{equation}

\item \textbf{Matrix notation:}
Our $S^{n-1}$ is contained in $\mathbb{R}^n$. There is one more nice property of spherical simplices, compared to the Euclidean ones.
In coordinates they are represented not as $(n+1)\times n$ matrices, but as $n \times n $ matrices.

Let $O_u, O_v$ be the circumcenters of $U, V$. Since $\angle (O_u, O_v) = \alpha,$ there exists an orthogonal basis $\mathcal{E} \:= \{e_1, \dots, e_n\}$, such that 
\begin{equation*}
\begin{cases}
O_u = e_1 \\
O_v = e_1\cos\alpha+e_2\sin\alpha.
\end{cases}
\end{equation*}

Define an $n \times n$ matrix:
\begin{equation*}
A:= \begin{pmatrix}
  \cos \alpha & \sin \alpha  & \cdots & 0 \\
  -\sin \alpha & \cos \alpha  & \cdots & 0 \\
  \vdots & \vdots  & \ddots & \vdots \\
  0 & 0 & \cdots & 1 \\
\end{pmatrix},
\end{equation*}
it is orthogonal. So $A\mathcal{E} = \{Ae_1, \dots, Ae_n\}$ is also an orthogonal basis. Also, $Ae_1 = O_v.$

Let the $i$-th vertex of $U$ have coordinates $(u_{i1},\dots,u_{in})$ in the basis $\mathcal{E}$.
Define an $n \times n$ matrix $U$ with each column equal to the coordinates of one vertex of simplex $U$.
So, 

\begin{equation*}
U = \begin{pmatrix}
  u_{11} &  \cdots & u_{n1} \\
  \vdots  & \ddots & \vdots \\
  u_{1n}  & \cdots & u_{nn} \\
\end{pmatrix}.
\end{equation*}
Define the similar matrix $V = \begin{pmatrix}
  v_{11} &  \cdots & v_{n1} \\
  \vdots  & \ddots & \vdots \\
  v_{1n}  & \cdots & v_{nn} \\
\end{pmatrix}$ containing the vertex coordinates of the second simplex, $V,$ in the basis $A\mathcal{E}.$
Then the coordinates of this simplex vertices in the basis $\mathcal{E}$ is $V$ times the inverse of $A:$

\begin{equation*}
A^{-1}V
\end{equation*}

\begin{itemize}
\item 
The mutual polarity of the two simplices, represented in the basis $\mathcal{E}$ leads to the fact that the matrix
\begin{equation*}
U^T \cdot A^{-1} V
\end{equation*}
is a diagonal matrix with all elements positive.
\end{itemize}

From the definition of the circumcenter:
\begin{equation}
\begin{cases}
u_{i1} = \cos \beta, \quad \forall i = 1, \dots, n\\
v_{i1} = \cos \gamma, \quad \forall i = 1, \dots, n,
\label{eq:cosbeta}
\end{cases}
\end{equation}
and so
\begin{equation}
\begin{cases}
u_{i2}^2 + \dots + u_{in}^2 = \sin^2 \beta, \quad \forall i = 1, \dots, n\\
v_{i2}^2 + \dots + v_{in}^2 = \sin^2 \gamma, \quad \forall i = 1, \dots, n,
\label{eq:sinbeta}
\end{cases}
\end{equation}
so 
\begin{equation}
\begin{cases}
b^2u_{i1}^2 -u_{i2}^2 - \dots - u_{in}^2 = \tan^2\beta\cos^2\beta-\sin^2 \beta = 0, \quad \forall i = 1, \dots, n\\
c^2v_{i1}^2 -v_{i2}^2 - \dots - v_{in}^2 = \tan^2\gamma\cos^2\gamma-\sin^2 \gamma = 0, \quad \forall i = 1, \dots, n.
\label{eq:btan}
\end{cases}
\end{equation}

Define the following diagonal $n \times n$ matrices:
\begin{equation*}
B:= \begin{pmatrix}
  b & 0  & \cdots & 0 \\
  0 & 1  & \cdots & 0 \\
  \vdots & \vdots  & \ddots & \vdots \\
  0 & 0 & \cdots & 1 \\
\end{pmatrix}, \quad 
C:= \begin{pmatrix}
  c & 0  & \cdots & 0 \\
  0 & 1  & \cdots & 0 \\
  \vdots & \vdots  & \ddots & \vdots \\
  0 & 0 & \cdots & 1 \\
\end{pmatrix}, \quad
J:= \begin{pmatrix}
  1 & 0  & \cdots & 0 \\
  0 & -1  & \cdots & 0 \\
  \vdots & \vdots  & \ddots & \vdots \\
  0 & 0 & \cdots & -1 \\
\end{pmatrix}.
\end{equation*}

\begin{itemize}
\item Condition \eqref{eq:btan} leads to the fact the matrices
\begin{equation}
U^T \cdot JB^2 \cdot U, \quad  V^T \cdot JC^2 \cdot V
\label{eq:polarity}
\end{equation}
have all diagonal elements equal 0. We can also write these matrices as 
\begin{equation*}
U^TB \cdot J \cdot BU, \quad  V^TC \cdot J \cdot CV.
\end{equation*}

\end{itemize}

Let's now rewrite the matrix from the polarity condition:
\begin{equation}
U^T A^{-1} V = U^T B B^{-1 }A^{-1} C^{-1}C V = U^T B \cdot \left(CAB \right) ^{-1}\cdot C V.
\label{eq:pause}
\end{equation}

In order to continue writing this expansion, we will focus on the matrix  $CAB $ and its so-called
Lorentz singular value decomposition.

\item \textbf{Lorentz singular value decomposition:}
By construction 
\begin{equation*}
CAB = \begin{pmatrix}
  bc \cos \alpha & c \sin \alpha  & \cdots & 0 \\
  -b \sin \alpha & \cos \alpha   & \cdots & 0 \\
  \vdots & \vdots  & \ddots & \vdots \\
  0 & 0 & \cdots & 1 \\
\end{pmatrix}.
\end{equation*}

The Lorentz singular value decomposition which we would like to apply to its top-left $2 \times 2$ corner is a representation in the form $\begin{pmatrix}
 \cosh t & \sinh t \\
 \sinh t & \cosh t \\
\end{pmatrix} 
\begin{pmatrix}
 K & 0 \\
 0 & L \\
\end{pmatrix}
\begin{pmatrix}
 \cosh s & \sinh s \\
 \sinh s & \cosh s \\
\end{pmatrix}, \quad K,L,s,t \in \mathbb{R}.$

Notice that not every $2 \times 2$ matrix has this form, for example 
$\begin{pmatrix}
 0 & 1 \\
 1 & 0 \\
\end{pmatrix}$ hasn't. Indeed, if it has, then this would imply that 
\begin{equation*}
\begin{pmatrix}
 0 & 1 \\
 1 & 0 \\
\end{pmatrix}
 = 
\begin{pmatrix}
 \cosh \tau & \sinh \tau \\
 \sinh \tau & \cosh \tau \\
\end{pmatrix} 
\begin{pmatrix}
 A & 0 \\
 0 & B \\
\end{pmatrix}
\begin{pmatrix}
 \cosh \sigma & \sinh \sigma \\
 \sinh \sigma & \cosh \sigma \\
\end{pmatrix}, \text{ so}
\end{equation*}

\begin{equation*}
\begin{pmatrix}
 A & 0 \\
 0 & B \\
\end{pmatrix}
 = 
\begin{pmatrix}
 \cosh \tau & -\sinh \tau \\
 -\sinh \tau & \cosh \tau \\
\end{pmatrix} 
\begin{pmatrix}
 0 & 1 \\
 1 & 0 \\
\end{pmatrix}
\begin{pmatrix}
 \cosh \sigma & -\sinh \sigma \\
 -\sinh \sigma & \cosh \sigma \\
\end{pmatrix} = 
\begin{pmatrix}
 -\sinh (\tau+\sigma) &  \cosh (\tau+\sigma) \\
\cosh (\tau+\sigma)&  -\sinh (\tau+\sigma) \\
\end{pmatrix} ,
\end{equation*}
which cannot be true for real $\tau, \sigma.$

Akin to the regular SVD, the Lorentz singular value decomposition (described in \cite{Lorentz}) transforms a given matrix into a diagonal one not by multiplying it by orthogonal matrices, but instead by multiplying it by $J-$orthogonal matrices. 
The concept has been thoroughly investigated by V. Šego, as presented in \cite{vedran-sego}. We suggest \cite{vedran-sego}, \cite{Lorentz2}, \cite{topo}, \cite{jort} for further details, and here we will simply present the exact decomposition without any use of literature. Let's prove a lemma:

\begin{lemma}
For any $X,Y,Z,T \in \mathbb{R}, $ such that
\begin{equation*}
\begin{cases}
X-T > |Z-Y|>0\\
X+T > |Z+Y|>0\\
XT-ZY > 0
\end{cases}
\end{equation*}
there exist $K, L, t, s \in \mathbb{R},$ such that
\begin{equation*}
\begin{pmatrix}
 X & Y \\
 Z & T \\
\end{pmatrix}
 = \begin{pmatrix}
 \cosh t & \sinh t \\
 \sinh t & \cosh t \\
\end{pmatrix} 
\begin{pmatrix}
 K & 0 \\
 0 & L \\
\end{pmatrix}
\begin{pmatrix}
 \cosh s & \sinh s \\
 \sinh s & \cosh s \\
\end{pmatrix}, \text{ and}
\end{equation*}

\begin{equation*}
\begin{cases}
K - L = \sqrt{(X-T)^2 - (Z-Y)^2}\\
K>L>0\\
\end{cases}.
\end{equation*}

\end{lemma}

\begin{proof}
There exist $K,L \in \mathbb{R}$ such that 
\begin{equation}
\begin{cases}
K + L = \sqrt{(X+T)^2 - (Z+Y)^2}\\
K - L = \sqrt{(X-T)^2 - (Z-Y)^2}.
\label{eq:sqrts}
\end{cases}
\end{equation}

Choose $t,s,$ such that 
\begin{equation}
\begin{cases}
t+s = \arcsinh \frac{Z+Y}{K+L}\\
t-s = \arcsinh \frac{Z-Y}{K-L}
\label{eq:arcsinh}
\end{cases}
\end{equation}.

From equation \eqref{eq:arcsinh} we have 
\begin{equation}
\begin{cases}
Z+Y = \sinh (t+s)(K+L)\\
Z-Y = \sinh (t-s)(K-L) \label{eq:sinh}.
\end{cases}
\end{equation}

From $X+T>0,X-T>0,$ and from equation \eqref{eq:sqrts}:
\begin{equation*}
\begin{cases}
X+T = \sqrt{(K+L)^2 + (Z+Y)^2}\\
X-T = \sqrt{(K-L)^2 + (Z-Y)^2}.
\end{cases}
\end{equation*}

From equation \eqref{eq:sinh} and since $K+L>0,K-L>0,$ we get:
\begin{equation*}
\sqrt{(K+L)^2 + (Z+Y)^2} = \sqrt{(K+L)^2 +  \sinh^2 (t+s)(K+L)^2} = \cosh(t+s) |K+L| = \cosh(t+s) (K+L),
\end{equation*}
\begin{equation*}
\sqrt{(K-L)^2 + (Z-Y)^2} = \sqrt{(K-L)^2 +  \sinh^2 (t-s)(K-L)^2} = \cosh(t-s) |K-L| = \cosh(t-s) (K-L).
\end{equation*}
So,
\begin{equation}
\begin{cases}
X+T = \cosh(t+s) (K+L)\\
X-T = \cosh(t-s) (K-L). \label{eq:cosh}
\end{cases}
\end{equation}

Now, combining equations \eqref{eq:sinh} and \eqref{eq:cosh} we get:
\begin{align*}
&\begin{pmatrix}
 \cosh t & \sinh t \\
 \sinh t & \cosh t \\
\end{pmatrix} 
\begin{pmatrix}
 K & 0 \\
 0 & L \\
\end{pmatrix}
\begin{pmatrix}
 \cosh s & \sinh s \\
 \sinh s & \cosh s \\
\end{pmatrix} = \\
=&\begin{pmatrix}
 \cosh t & \sinh t \\
 \sinh t & \cosh t \\
\end{pmatrix} 
\left(
\frac{K+L}{2}
\begin{pmatrix}
 1 & 0 \\
 0 & 1 \\
\end{pmatrix}
+
\frac{K-L}{2}
\begin{pmatrix}
 1 & 0 \\
 0 & -1 \\
\end{pmatrix}
\right)
\begin{pmatrix}
 \cosh s & \sinh s \\
 \sinh s & \cosh s \\
\end{pmatrix} = \\
=\frac{K+L}{2}
&\begin{pmatrix}
 \cosh t & \sinh t \\
 \sinh t & \cosh t \\
\end{pmatrix} 
\begin{pmatrix}
 \cosh s & \sinh s \\
 \sinh s & \cosh s \\
\end{pmatrix}+
\frac{K-L}{2}
\begin{pmatrix}
 \cosh t & -\sinh t \\
 \sinh t & -\cosh t \\
\end{pmatrix} 
\begin{pmatrix}
 \cosh s & \sinh s \\
 \sinh s & \cosh s \\
\end{pmatrix} = \\
=\frac{K+L}{2}
&\begin{pmatrix}
 \cosh (t + s) & \sinh (t + s) \\
 \sinh (t + s) & \cosh (t + s) \\
\end{pmatrix}+
\frac{K-L}{2}
\begin{pmatrix}
 \cosh (t - s) & -\sinh (t - s) \\
 \sinh (t - s) & -\cosh (t - s) \\
\end{pmatrix} = \\
=&\begin{pmatrix}
\frac{X+T}{2} & \frac{Z+Y}{2} \\
 \frac{Z+Y}{2} & \frac{X+T}{2} \\
\end{pmatrix}
+
\begin{pmatrix}
\frac{X-T}{2} & -\frac{Z-Y}{2} \\
 \frac{Z-Y}{2} & -\frac{X-T}{2} \\
\end{pmatrix} = 
\begin{pmatrix}
X& Y \\
 Z & T \\
\end{pmatrix}.
\end{align*}

From equation \eqref{eq:sqrts} we get $K-L >0, K>0.$

Also, 
\begin{equation*}
XT-ZY = 
\det
\begin{pmatrix}
X& Y \\
 Z & T \\
\end{pmatrix} = 
\det
\begin{pmatrix}
K & 0 \\
 0 & L \\
\end{pmatrix} = KL, \text{ so }
\end{equation*}
\begin{equation*}
L>0.
\end{equation*}
\end{proof}

Let's check the Lemma's assumptions for the matrix
\begin{equation*}
\begin{pmatrix}
 X & Y \\
 Z & T \\
\end{pmatrix}:= 
\begin{pmatrix}
  bc \cos \alpha & c \sin \alpha   \\
  -b \sin \alpha & \cos \alpha \\
\end{pmatrix}.
\end{equation*}
From inequality \eqref{eq:ineq1} it follows that  $\tan\left(\beta+\gamma-\frac{\pi}{2}\right) > \tan\alpha,$ or 
$-\frac{1}{\tan(\beta + \gamma)}> \tan\alpha,$ or

\begin{equation*}
\frac{bc-1}{b+c}> \tan\alpha.
\end{equation*}

From inequality \eqref{eq:ineq2}, we have $\cos \alpha>0, $ and so

\begin{equation*}
(bc-1) \cos \alpha > (b+c) \sin \alpha.
\end{equation*}

So $X-T>|Z-Y|.$ Also, as $Z$ and $Y$ have different signs:

\begin{equation*}
X+T > X-T > |Z-Y| > |Z+Y|.
\end{equation*}

Also
\begin{equation*}
XT-ZY = bc\cos^2\alpha+bc\sin^2\alpha =bc>0.
\end{equation*}

So Lemma holds, and it leads to:
\begin{equation*}
K - L = \sqrt{(X-T)^2 - (Z-Y)^2} = \sqrt{(bc-1)^2\cos^2\alpha - (b+c)^2\sin^2\alpha}, \text{ and}
\end{equation*}
\begin{equation*}
K>L>0.
\end{equation*}
\item \textbf{Replacing two cones with one:}
 
 Define the following $n \times n$ matrices:
\begin{equation*}
D:= \begin{pmatrix}
  K & 0  & \cdots & 0 \\
  0 & L  & \cdots & 0 \\
  \vdots & \vdots  & \ddots & \vdots \\
  0 & 0 & \cdots & 1 \\
\end{pmatrix}, \quad
\sqrt{D^{-1}}:= \begin{pmatrix}
  \frac{1}{\sqrt{K}} & 0  & \cdots & 0 \\
  0 & \frac{1}{\sqrt{L}}  & \cdots & 0 \\
  \vdots & \vdots  & \ddots & \vdots \\
  0 & 0 & \cdots & 1 \\
\end{pmatrix};
\end{equation*}
\begin{equation*}
\forall x \in \mathbb{R},\quad
R_x:= \begin{pmatrix}
  \cosh x & \sinh x  & \cdots & 0 \\
 \sinh x & \cosh x  & \cdots & 0 \\
  \vdots & \vdots  & \ddots & \vdots \\
  0 & 0 & \cdots & 1 \\
\end{pmatrix}.
\end{equation*}

We have shown above, that

\begin{equation*}
CAB = R_t D R_s,
\end{equation*}
and that $\sqrt{D^{-1}}$ has its diagonal entries real (and positive).

Let's continue expanding equation \eqref{eq:pause}:

\begin{align*}
U^T A^{-1} V = U^T B \left(CAB \right) ^{-1}C V = U^T B \left(R_t D R_s\right) ^{-1}C V = U^TBR_{-s}D^{-1}R_{-t}CV = \\
 = (U^TBR_{-s}\sqrt{D^{-1}})(\sqrt{D^{-1}}R_{-t}CV) = (\sqrt{D^{-1}}R_{-s}BU)^T(\sqrt{D^{-1}}R_{-t}CV).
\end{align*}

From $\cosh^2 x - \sinh^2 x = 1$ it follows that
\begin{equation*}
R_x JR_x = J, \quad \forall  x \in \mathbb{R}.
\end{equation*}

Let's rewrite the matrices in equation \eqref{eq:polarity}:

\begin{align*}
U^T \cdot JB^2 \cdot U & = U^T B\cdot J\cdot B U = U^T BR_{-s}\cdot J\cdot R_{-s}BU = U^TBR_{-s}\sqrt{D^{-1}}JD\sqrt{D^{-1}}R_{-s}BU = \\
& = (U^TBR_{-s}\sqrt{D^{-1}})JD(\sqrt{D^{-1}}R_{-s}BU) = (\sqrt{D^{-1}}R_{-s}BU)^TJD(\sqrt{D^{-1}}R_{-s}BU)\text{, and}
\end{align*}

\begin{equation*}
V^T \cdot JC^2 \cdot V = (\sqrt{D^{-1}}R_{-t}CV)^TJD(\sqrt{D^{-1}}R_{-t}CV).
\end{equation*}

Finally, let 
\begin{equation*}
\begin{cases}
U':= \sqrt{D^{-1}}R_{-s}BU\\
V':= \sqrt{D^{-1}}R_{-t}CV,
\end{cases}
\end{equation*}
which gives
\begin{equation*}
\begin{cases}
U^T A^{-1} V = U'^{T}V' \\
U^T \cdot JB^2 \cdot U = U'^{T} \cdot JD \cdot U'\\
V^T \cdot JC^2 \cdot V = V'^{T} \cdot JD \cdot V'.
\end{cases}
\end{equation*}

Let's notice that 
\begin{equation*}
\Tr (JD) = K - L - (n-2) = \sqrt{(bc-1)^2\cos^2\alpha - (b+c)^2\sin^2\alpha} - (n-2).
\end{equation*}

We will prove that $\Tr (JD) $ is non-negative, and this will finalise the theorem. However, to be able to operate with
$U', V'$ we should prove one more condition concerning them.

\item \textbf{Columns of \( U' \), \( V' \) belong to the same semi-space.}
As defined above,
\begin{equation*}
  U' = \sqrt{D^{-1}}R_{-s}B \cdot U, \quad V' = \sqrt{D^{-1}}R_{-t}C \cdot V.
\end{equation*}
Let's prove that the top row of the matrix $U'$ contains only positive elements.
Let's expand \( \sqrt{D^{-1}}R_{-s}B \), its only non-trivial elements are in the top-left \( 2 \times 2 \) corner.
\begin{align*}
  \sqrt{D^{-1}}R_{-s}B & = \begin{pmatrix}
  \frac{1}{\sqrt{K}} & 0 & \cdots & 0 \\
  0 & \frac{1}{\sqrt{L}} & \cdots & 0 \\
  \vdots & \vdots & \ddots & \vdots \\
  0 & 0 & \cdots & 1
  \end{pmatrix}
  \begin{pmatrix}
  \cosh s & -\sinh s & \cdots & 0 \\
  -\sinh s & \cosh s & \cdots & 0 \\
  \vdots & \vdots & \ddots & \vdots \\
  0 & 0 & \cdots & 1
  \end{pmatrix}
  \begin{pmatrix}
  b & 0 & \cdots & 0 \\
  0 & 1 & \cdots & 0 \\
  \vdots & \vdots & \ddots & \vdots \\
  0 & 0 & \cdots & 1
  \end{pmatrix} =\\
  & = \begin{pmatrix}
  \frac{b \cosh s}{\sqrt{K}} & -\frac{\sinh s}{\sqrt{K}} & \cdots & 0 \\
  -\frac{b \sinh s}{\sqrt{L}} & \frac{\cosh s}{\sqrt{L}} & \cdots & 0 \\
  \vdots & \vdots & \ddots & \vdots \\
  0 & 0 & \cdots & 1
  \end{pmatrix},
\end{align*}
so
\begin{align*}
  \sqrt{D^{-1}}R_{-s}B \cdot U & = \begin{pmatrix}
  \frac{b \cosh s}{\sqrt{K}} & -\frac{\sinh s}{\sqrt{K}} & \cdots & 0 \\
  -\frac{b \sinh s}{\sqrt{L}} & \frac{\cosh s}{\sqrt{L}} & \cdots & 0 \\
  \vdots & \vdots & \ddots & \vdots \\
  0 & 0 & \cdots & 1
  \end{pmatrix}
  \begin{pmatrix}
  u_{11} & \cdots & u_{1n} \\
  \vdots & \ddots & \vdots \\
  u_{n1} & \cdots & u_{nn}
  \end{pmatrix}=
 \\
  & = 
  \begin{pmatrix}
  \frac{b \cosh s}{\sqrt{K}} u_{11} - \frac{\sinh s}{\sqrt{K}} u_{12} & \cdots & \frac{b \cosh s}{\sqrt{K}} u_{n1} - \frac{\sinh s}{\sqrt{K}} u_{n2}\\
  \vdots & \vdots & \vdots
  \end{pmatrix}.
\end{align*}

Let's prove that each element of the top row of this matrix is larger than \( 0 \). First, from equations \eqref{eq:cosbeta} and \eqref{eq:sinbeta}:
\begin{equation*}
  bu_{i1} = \tan \beta \cos \beta = \sin \beta > |u_{i2}|, \quad \forall i \in \{1, \ldots, n\}.
  \end{equation*}
Second,
\begin{equation*}
  \cosh s > |\sinh s|.
\end{equation*}
Multiplying the two, and dividing by \( \sqrt{K} \):
\begin{equation*}
  \frac{b \cosh s}{\sqrt{K}} u_{i1} > \left|\frac{\sinh s}{\sqrt{K}} u_{i2}\right|,
\end{equation*}
which gives that the top row of \( \sqrt{D^{-1}}R_{-s}B \cdot U \) contains only positive elements. Similarly, the top row of \(V' = \sqrt{D^{-1}}R_{-t}C \cdot V \) contains only positive elements.

\item \textbf{The trace is non-negative:}
We will prove a Lemma that finalises the Theorem statement:
\begin{lemma}
Let $n \geqslant 3$ be an integer, let $U,V$ be $n \times n$ matrices, with the first row containing only positive elements.

Let
\begin{equation*}
M = 
\begin{pmatrix}
  m_1 & 0  & \cdots & 0 \\
  0 & -m_2  & \cdots & 0 \\
  \vdots & \vdots  & \ddots & \vdots \\
  0 & 0 & \cdots & -m_n \\
\end{pmatrix}
\end{equation*}
be an $n \times n$ diagonal matrix, where all $m_i > 0.$ If the matrices 
$U^{T}MU, V^{T}MV$
have all diagonal elements equal to 0, and 
$U^{T}V$
is a diagonal matrix with all its diagonal elements positive, then 
\begin{equation*}
\Tr M \geqslant 0.
\end{equation*}
In the case where $n = 3$, the inequality becomes an equality.
\end{lemma}
\begin{proof}
Consider a bilinear form $M:\mathbb{R}^n \times \mathbb{R}^n \to \mathbb{R},$ defined by the matrix $M:$
\begin{equation*}
M(x,y):=m_1x_1y_1-m_2x_2y_2 -\dots-m_nx_ny_n.
\end{equation*}
Denote the columns of $U, V$ as \( \overline{u}_i, \overline{v}_i, i = 1, \dots, n.\) Then
\begin{equation*}
 M(\overline{u}_i,\overline{u}_i)=M(\overline{v}_i,\overline{v}_i)=0, \quad \forall i = 1, \dots, n.
\end{equation*}
First let's show that $M(\overline{u}_i,\overline{v}_i)\geqslant 0.$ Let $\overline{u}_i=(x_1,\dots,x_n), \overline{v}_i=(y_1,\dots,y_n),$ then
\begin{align*}
(m_1x_1y_1)^2=m_1x_1^2 \cdot m_1 y_1^2=(m_2x_2^2+\dots+m_nx_n^2)(m_2y_2^2+\dots+m_ny_n^2) \geqslant (m_2x_2y_2+\dots+m_nx_ny_n)^2.
\end{align*}
So, as $m_1, x_1, y_1 > 0,$ we have $M(\overline{u}_i,\overline{v}_i)\geqslant 0.$

Now, \( U^TV \) is an invertible diagonal matrix (all its diagonal elements are positive), so both \( U \) and \( V \) are invertible. Let's choose some unit vector $e.$ We can write:
\begin{equation}
\overline{e} = x_1\overline{u}_1+\dots+x_n\overline{u}_n, \text{ for some }x_i \in \mathbb{R}, i = 1,\dots,n.
\label{eq:g}
\end{equation}
Then, multiplying both sides by \( \overline{v}_j \) for some \( j \):
\begin{equation*}
\overline{e}\overline{v}_j = x_j\overline{u}_j\overline{v}_j, \text{ or}
\end{equation*}
\begin{equation*}
x_j = \frac{\overline{e}\overline{v}_j}{\overline{u}_j\overline{v}_j}.
\end{equation*}
Now, multiplying both sides of \eqref{eq:g} by \( \overline{e}, \) we get:
\begin{equation*}
1=\overline{e}^2 = x_1\overline{e}\overline{u}_1+\dots+x_n\overline{e}\overline{u}_n = \frac{\overline{e}\overline{u}_1 \cdot \overline{e}\overline{v}_1}{\overline{u}_1\overline{v}_1}+\dots+\frac{\overline{e}\overline{u}_n \cdot \overline{e}\overline{v}_n}{\overline{u}_n\overline{v}_n}.
\end{equation*}

This means that:

\begin{equation}
\Tr M = \frac{M(\overline{u}_1,\overline{v}_1)}{\overline{u}_1\overline{v}_1} + \dots + \frac{M(\overline{u}_n,\overline{v}_n)}{\overline{u}_n\overline{v}_n}.
\end{equation}

As $\overline{u}_i\overline{v}_i>0,$ we have $\Tr M \geqslant 0.$
\end{proof}

\item \textbf{In the case where $n = 3$ equality holds:} The equality in this case is equivalent to the equation \eqref{spherical}.





\end{enumerate}
\end{proof}
\section{An $n$-dimensional Grace-Danielsson Inequality}

\begin{corollary}
\label{cor:sphgd}
For a spherical simplex $U$ with $n+1$ vertices, let $\beta, \Gamma$ be its radii of the circumscribed and inscribed spheres, and $\alpha$ be the distance between their centers. Let
\begin{align*}
R:= \tan\beta, \quad r:= \tan\Gamma, \quad d:= \tan\alpha.
\end{align*}
Then
\begin{equation*}
(R-nr)(R+(n-2)r)> d^2.
\end{equation*}
\end{corollary}
\begin{proof}
Let $\gamma:= \frac{\pi}{2}-\Gamma,$ so $\tan\gamma = \frac{1}{\tan\Gamma} = \frac{1}{r}.$ As $\gamma$ is the circumscribed sphere radius of a simplex polar to $U,$ we from Theorem \ref{th:main} have:
\begin{equation*}
\sqrt{ \left(R\cdot\frac{1}{r}-1\right)^2  \cos ^ 2 \alpha - \left(R + \frac{1}{r}\right)^2 \sin^2 \alpha} \geqslant n - 1.
\end{equation*}

As $d = \tan\alpha,$ we have:
\begin{align*}
\cos^2\alpha = \frac{1}{1+d^2}, \quad \sin^2\alpha = \frac{d^2}{1+d^2},
\end{align*}
and so
\begin{equation*}
\sqrt{ \left(R\cdot\frac{1}{r}-1\right)^2  \frac{1}{1+d^2} - \left(R + \frac{1}{r}\right)^2 \frac{d^2}{1+d^2}} \geqslant n - 1.
\end{equation*}

Taking the square and multiplying both parts by $r^2(1+d^2)$ we get:
\begin{equation*}
(R-r)^2-(1+Rr)^2d^2\geqslant(n-1)^2(1+d^2)r^2.
\end{equation*}
As $1+Rr>1, 1+d^2>1$ it leads to:
\begin{equation*}
(R-r)^2-d^2>(n-1)^2r^2, \text{ or}
\end{equation*}
\begin{equation*}
(R-nr)(R+(n-2)r)> d^2.
\end{equation*}
\end{proof}

 Notice that Corollary
\ref{cor:sphgd} is proven for a spherical simplex in $S^{n}$ of radius 1, but actually holds for a spherical simplex on a sphere of any radius. We will now make these spheres very large.

\begin{theorem}
\label{th:ecproof}
 \textbf{(Egan conjecture; an $n$-dimensional Grace-Danielsson inequality.)} 
Consider a Euclidean simplex $U$ in $\mathbb{R}^n$ with $n+1$ vertices. Let the radii of its circumscribed and inscribed spheres
be $R$ and $r$ respectively, and the distance between their centers be $d.$ Then the following inequality holds:
\begin{equation*}
(R-nr)(R+(n-2)r)\geqslant d^2.
\end{equation*}
\end{theorem}
\begin{proof}
\begin{enumerate}
\item \textbf{Construct an $n$-dimensional sphere:}
 Let the circumscribed and inscribed spheres of  $U$ have centers $O$ and $I$ respectively.
Let's consider such space $\mathbb{R}^{n+1},$ that $\left<U\right>$ is a hyperplane in it.
Set the origin at $O$.
Let the vertices of the simplex $U$ be $u_1, \dots, u_{n+1} \in \left< U\right> \subset \mathbb{R}^{n+1}.$ So $|u_i| = R, i = 1,\dots,n+1.$
Choose a unit vector $\overline{e},$ such that $\overline{e} \perp  \left<U\right>.$

Let point $P$ be such that $\overline{OP} = H\overline{e},\quad H\to\infty.$

Let $u'_i:= u_i+\overline{PO}.$
As $\overline{OP} \perp u_i,$ all distances from $P$ to $u_i$ are the same:

\begin{equation*}
|u_i'| = \sqrt{H^2+R^2}, \quad \forall i = 1,\dots,n+1.
\end{equation*}

Now consider an $n$-dimensional sphere $S_H,$ with the center $P$ and radius $\sqrt{H^2+R^2}.$ It contains all $u_1,\dots,u_{n+1}$.

Let's call $\{u'_i\}$ on the sphere $S_H$ the spherical simplex $U'$ (it is a non-degenerate spherical simplex). We will call its centers of circumscribed and inscribed spheres $O_H, I_H.$ Their radii $\beta_H, \Gamma_H.$ We will call the spherical distance between $O_H$ and $I_H$ to be $\alpha_H.$ Also:
\begin{equation*}
R_H:=\tan \beta_H, \quad r_H:=\tan\Gamma_H, \quad d_H:=\tan\alpha_H.
\end{equation*}

\item\textbf{$O_H$ converges to $O$:}
By construction we can write:
\begin{align*}
|O_HO| & = \sqrt{H^2+R^2}-H \to 0 \\
R_H & = \tan\beta_H = \frac{R}{H},
\end{align*}
which also leads to 
\begin{equation}
R_H \sqrt{H^2+R^2} \to R \text{, as } H \to \infty.
\label{ineq:conv1}
\end{equation}

\item\textbf{$I_H$ converges to $I$:}
Denote all full-dimensional faces of $U$ as $F_i,$ and their $n-1$-dimensional volumes as $V_i:$
\begin{equation*}
V_i:= \Vol_{n-1}(F_i).
\end{equation*}
Point $I$ belongs to $\left<U\right>$ and is equidistant from all $F_i.$
We can write the incenter's barycentric coordinates for a Euclidean simplex $U$ with $n+1$ vertex:

\begin{equation*}
I 
= \frac{\sum\limits_{i = 1}^{n+1}V_i u_i}{\sum\limits_{i = 1}^{n+1} V_i}.
\end{equation*}

Point $I_H$ belongs to the sphere $S_H$ and is equidistant from all of $\left<F_i,P\right>$ ($n$-dimensional hyperplanes).
Consider an $n+2-$vertex Euclidean simplex $\{U,P\}$, and the point $L_H$ at the face $U,$ which is equidistant from all of $\left<F_i,P\right>.$
It is the intersection of the bisector of angle $P$ with face $U.$ We can write the barycentric coordinates of $L_H$:
\begin{equation*}
L_H = \frac{\sum\limits_{i = 1}^{n+1} \Vol_{n}(\Conv(F_i,P)) u_i}{\sum\limits_{i = 1}^{n+1} \Vol_n(\Conv(F_i,P))}.
\end{equation*}
Denoting the distances from $O$ to $F_i$ by $d_i$ we can continue:
\begin{equation*}
\Vol_{n}(\Conv(F_i,P)) = \frac{\Vol_{n-1}(F_i) \sqrt{H^2+d_i^2}}{n} = \frac{V_i \sqrt{H^2+d_i^2}}{n}
\end{equation*}

\begin{equation}
L_H = \frac{\frac{1}{n}\sum\limits_{i = 1}^{n+1} V_i \sqrt{H^2+d_i^2} u_i}{\frac{1}{n}\sum\limits_{i = 1}^{n+1} V_i \sqrt{H^2+d_i^2}} = 
\frac{\sum\limits_{i = 1}^{n+1} V_i u_i\sqrt{1+\frac{d_i^2}{H^2}} }{\sum\limits_{i = 1}^{n+1} V_i\sqrt{1+\frac{d_i^2}{H^2}} }\to I, \text{ as } H \to \infty.
\label{eq:lhtoi}
\end{equation}

We can estimate $|PL_H| \geqslant H.$ As $|PI_H| = \sqrt{H^2 + R^2},$ and $L_H \in [PI_H],$ the following holds:

\begin{equation}
|L_HI_H| \leqslant \sqrt{H^2+R^2} - H \to 0, \text{ as } H \to \infty.
\label{eq:lhtoih}
\end{equation}
Thus, from the convergence conditions \eqref{eq:lhtoi}, \eqref{eq:lhtoih}:
\begin{equation*}
I_H \to I, \text{ as } H \to \infty.
\end{equation*}

We now want to write a convergence condition also for $r_H.$ From $L_H$ on bisector:

\begin{equation*}
\dist(L_H,\left<F_i,P\right>)=\frac{(n+1)\Vol_{n+1}(\{U,P\})}{\sum\limits_{i=1}^{n+1} \Vol_n(\Conv(F_i,P))}=
\frac{\Vol_{n}(U)H}{\frac{H}{n}\sum\limits_{i = 1}^{n+1} V_i\sqrt{1+\frac{d_i^2}{H^2}} } \to
\frac{n\Vol_{n}(U)}{\sum\limits_{i = 1}^{n+1} V_i  }=r.
\end{equation*}
From $|L_HI_H|\to 0$ we also have 
\begin{equation*}
\dist(I_H,\left<F_i,P\right>)\to r.
\end{equation*}
As $|PI_H|=\sqrt{H^2+R^2},$
\begin{equation*}
\dist(I_H,\left<F_i,P\right>)=\sqrt{H^2+R^2}\sin \Gamma_H, \text{ and }
\end{equation*}
\begin{equation*}
\sqrt{H^2+R^2}\sin\Gamma_H\to r \text{, so}
\end{equation*}
\begin{equation}
r_H\sqrt{H^2+R^2}=\sqrt{H^2+R^2}\tan \Gamma_H \to r.
\label{ineq:conv2}
\end{equation}

\item\textbf{Euclidean inequality is a limit of spherical inequalities:}

From Corollary
\ref{cor:sphgd}:
\begin{equation*}
(R_H-nr_H)(R_H+(n-2)r_H) > d_H^2, \quad \forall H.
\end{equation*}
Multiplying by $H^2+R^2$ we get
\begin{equation*}
\left(R_H\sqrt{H^2+R^2}-nr_H\sqrt{H^2+R^2}\right)\left(R_H\sqrt{H^2+R^2}+(n-2)r_H\sqrt{H^2+R^2}\right) >(H^2+R^2)d_H^2, \quad \forall H.
\end{equation*}
From the convergence conditions \eqref{ineq:conv1}, \eqref{ineq:conv2} the left part converges to $(R-nr)(R+(n-2)r).$

Let's estimate the right part:
\begin{equation*}
d_H\sqrt{H^2+R^2} > |O_HI_H|.
\end{equation*}
As $O_H \to O, I_H \to I, $ then $ |O_HI_H| \to |OI|=d.$ 

So, 
\begin{equation*}
(R-nr)(R+(n-2)r)\geqslant d^2.
\end{equation*}

\end{enumerate}
\end{proof}
\bibliography{root} 
\bibliographystyle{ieeetr}

\end{document}